\documentclass[12pt]{amsart}
\usepackage{main}

\def\grad{\nabla}
\def\div{\operatorname{div}}

\def\bar{\overline}

\def\tilde{\widetilde}

\def\F{\mathcal F}

\def\cydot{\leavevmode\raise.4ex\hbox{.}}

\title[Lipschitz stable reconstruction formula]{A Lipschitz stable reconstruction formula for the inverse problem for the wave equation}
\author{Shitao Liu}
\author{Lauri Oksanen}
\address{Department of Mathematics and Statistics, University of Helsinki, P.O. Box 68 FI-00014}
\email{shitao.liu@helsinki.fi}
\email{lauri.oksanen@helsinki.fi}
\begin{document}

\begin{abstract}
We consider the problem to reconstruct a wave speed $c \in C^\infty(M)$ in a domain $M \subset \R^n$ from 
acoustic boundary measurements modelled by the hyperbolic Dirichlet-to-Neumann map $\Lambda$.
We introduce a reconstruction formula for $c$
that is based on the Boundary Control method and incorporates features also from 
the complex geometric optics solutions approach. 
Moreover, we show that the reconstruction formula is locally Lipschitz stable
for a low frequency component of $c^{-2}$ under the assumption that
the Riemannian manifold $(M, c^{-2} dx^2)$ has a strictly convex function with no critical points.
That is, we show that for all bounded $C^2$ neighborhoods $U$ of $c$,
there is a $C^1$ neighborhood $V$ of $c$ and constants $C, R > 0$ such that 
\begin{align*}
|\F\ll(\tilde c^{-2} - c^{-2}\rr)(\xi)| \le C e^{2R |\xi|} \norm{\tilde \Lambda - \Lambda}_*,
\quad \xi \in \R^n,
\end{align*}
for all $\tilde c \in U \cap V$, where
$\tilde \Lambda$ is the Dirichlet-to-Neumann map corresponding to 
the wave speed $\tilde c$
and $\norm{\cdot}_*$ is a norm capturing certain regularity properties of 
the Dirichlet-to-Neumann maps.
\end{abstract}

\maketitle

\section{Introduction}

Let $M \subset \R^n$ be a compact set with nonempty interior and a smooth boundary $\p M$
and let $c \in C^\infty(M)$ be strictly positive.
We consider the wave equation on $M$,
\begin{align}
\label{eq_wave}
&\p_t^2 u(t,x) - c(x)^2 \Delta u(t, x) = 0, & (t,x) \in (0,\infty) \times M,
\\\nonumber& u(t,x) = f(t,x), & (t,x) \in (0,\infty) \times \p M,
\\\nonumber& u|_{t=0}(x) = 0, \quad \p_t u|_{t=0}(x)=0,  & x \in M.
\end{align}
Let us denote the solution of (\ref{eq_wave}) by $u^f(t,x) = u(t,x)$ and let $T > 0$. We define the operator 
\begin{align*}
\Lambda_{c,T} : f \mapsto \p_\nu u^f|_{(0,T) \times \p M},
\quad f \in C_0^{\infty}((0, T) \times \p M),
\end{align*}
where $\p_\nu$ is the Euclidean normal derivative on $\p M$.
Often we write $\Lambda_T := \Lambda_{c,T}$.
The operator $\Lambda_{T}$ models acoustic boundary measurements and is called the Dirichlet-to-Neumann operator. 
Let us denote by $C_+^\infty(M)$ the set of the strictly positive functions
in $C^\infty(M)$.
Then the inverse problem for the wave equation can be formulated as follows:
\begin{itemize}
\item[(IP)] Reconstruct $c \in C_+^\infty(M)$ given the operator $\Lambda_{c,T}$.
\end{itemize}

The finite speed of propagation for the wave equation 
(\ref{eq_wave}) gives a necessary condition
for $T$ in order to (IP) to have a unique solution. 
Indeed, if there is $x_0 \in M$ such that 
$T < 2 d(x_0, \p M)$,
where $d$ is the distance function of the Riemannian manifold $(M, c^{-2} dx^2)$,
then the measurements $\Lambda_{T}$
can not contain any information about $c(x_0)$.
Conversely, 
the problem (IP) is known to be uniquely solvable for 
$T$ strictly greater than the maximum of $2 d(x, \p M)$ for $x\in M$.
The global uniqueness can be proven either by using the Boundary Control (BC) method 
originated from \cite{Belishev1987} or by using the complex geometric optics (CGO) solutions originated from \cite{Sylvester1987}.
However, a typical application of the BC method depends on 
Tataru's unique continuation theorem \cite{Tataru1995}, whence 
only logarithmic type stability is expected for such an application.
The CGO solutions based approach 
is also typically limited to logarithmic type stability \cite{Mandache2001}.

Here we will introduce a global reconstruction method 
and prove that it is locally Lipschitz stable in the sense that we will describe below. 
The method is a modification of the BC method and employs also the 
harmonic exponential functions of the form
\begin{align*}
e^{i(\xi + i \eta) \cdot x/2},
\quad \xi, \eta \in \R^n, 
\ |\xi| = |\eta|,\ \xi \perp \eta,
\end{align*}
that are CGO solutions for the Euclidean Laplacian.   

H\"{o}lder stability with an exponent strictly better than $1/2$  
allows an inverse problem to be solved locally 
by the nonlinear Landweber iteration \cite{Hoop2011}.
Moreover, the convergence rate of the iteration is linear 
if and only if the problem is Lipschitz stable. 
Hence Lipschitz stability for (IP) would be important even without our explicit reconstruction method. 

H\"{o}lder type stability results for (IP) were first obtained in \cite{Stefanov1998, Stefanov2005},
and the best H\"{o}lder exponent available in the literature is $1/2$, see \cite{Bellassoued2011}. 
However, the H\"{o}lder exponent $1/2$ 
does not allow the convergence result \cite{Hoop2011} to be applied in a straightforward manner. Moreover, the technique in \cite{Bellassoued2011} does not give a global reconstruction method since it employs the geometric optics solutions corresponding to a fixed wave speed $c_0$, and requires $c_0$ to be known a priori. 

The stability result in \cite{Bellassoued2011} depends on the assumption that the Riemannian manifold $(M, c_0^{-2} dx^2)$ is simple. Similarly, our result depends on an assumption that we call {\em stable observability} (see Definition 1 below) and that is also of geometrical nature. 
Let us also point out that Stefanov and Uhlmann \cite{Stefanov2009a} have considered 
the linear inverse problem to recover the initial value of a solution $u$ to the wave equation 
in $\R^n$ given its trace on $\p \Omega \times (0,T)$ where $\Omega \subset \R^n$.
They prove that reconstruction of $u(0)$ supported in $\Omega$
can not be Lipschitz stable if there is a geodesic $\gamma$
such that $\gamma(0) \in \Omega$ and $\gamma([-T, T]) \cap \p \Omega = \emptyset$.
Thus it would be unexpected if the non-linear inverse problem (IP)
was locally Lipschitz stable without additional assumptions on the geometry $(M,c^{-2} dx^2)$.


\subsection{Statement of the main results}

We recall that the wave equation (\ref{eq_wave}) is said to be continuously observable from open $\Gamma \subset \p M$ in time $T > 0$ 
if there is $C_{obs} > 0$ such that
\begin{align}
\label{continuous_observability}
\norm{u(T)}_{H_0^1(M)} + \norm{\p_t u(T)}_{L^2(M)} \le C_{obs} \norm{\p_\nu u}_{L^2((0,T) \times \Gamma)}, 
\end{align}
where $u$ is a solution of the wave equation
\begin{align}
\label{eq_wave_w}
&\p_t^2 u - c^2 \Delta u = 0, &\text{in $(0, T) \times M$},
\\\nonumber &u|_{(0, T) \times \p M} = 0, &\text{in $(0, T) \times \p M$}.
\end{align}
The condition by Bardos, Lebeau and Rauch gives a geometric characterization 
of the continuous observability \cite{Bardos1992, Burq1997}.
In particular, 
if $\Gamma=\partial M$ and $(M, c^{-2}dx^2)$ is non-trapping
then the continuous observability (\ref{continuous_observability}) 
is valid. 
This is analogous with the condition in the above mentioned \cite{Stefanov2009a}.
We refer to \cite{Bardos1992} for the precise formulation of the geometric condition. 

\begin{definition}
Let $U \subset C_+^\infty(M)$. 
We say that the wave equations (\ref{eq_wave}) are {\em stably observable}
for $c \in U$, from open $\Gamma \subset \p M$ in time $T > 0$, 
if there is $C_{obs} > 0$ satisfying the following:
for all $c \in U$ the solutions 
of the wave equation (\ref{eq_wave_w})
satisfy the observability inequality (\ref{continuous_observability}).
\end{definition}
This stronger form of observability does not follow from the technique 
in \cite{Bardos1992} since the compactness-uniqueness argument
there does not bound the constant $C_{obs}$ in terms of the geometry $(M,c^{-2}dx^2)$.  
However, we will prove the following theorem.

\begin{theorem}
\label{thm_stable_for_semisimples}
Let $c \in C_+^\infty(M)$ and 
suppose that there is a strictly convex function $\ell \in C^3(M)$ 
with respect to the metric tensor $c^{-2} dx^2$, and that $\ell$ has no critical points.
Let $U \subset C_+^\infty(M)$ be bounded in $C^2(M)$
and let $\Gamma \subset \p M$ be a neighborhood of 
\begin{align*}
\{ x \in \p M;\ \grad \ell(x) \cdot \nu \ge 0 \}.
\end{align*}
Then there is a neighborhood $V$ of $c$ in $C^1(M)$ and $T > 0$ such that
the wave equations (\ref{eq_wave}) are stably observable
for the wave speeds in the set $U \cap V$, from $\Gamma$ in time $T$.
\end{theorem}

We will prove Theorem \ref{thm_stable_for_semisimples} 
in Section \ref{sec_stable_observability} also for anisotropic wave speeds
by first deriving a geometric Carleman estimate for the wave equation.
The main feature of the estimate is the absence of lower order terms,
and we will follow the tradition of this type of estimates, see \cite{Kazemi1993, Lavrentcprimeev1986, Tataru1996, Triggiani2002},
where the continuous observability was studied but the dependence of the constant $C_{obs}$ on the coefficients of the equation was not considered; and \cite{Zuazua2008}, where the dependence of $C_{obs}$ on the zeroth order coefficient of the equation was studied.

There are non-simple Riemannian manifolds $(M,g)$ that admit 
strictly convex functions with no critical points
(a trivial example being a non-convex subset of the Euclidean space).
See \cite{Gulliver2000} and \cite{Miller2002} for further discussion
on the relations between simplicity, the existence of a strictly convex function
and the characterization by Bardos, Lebeau and Rauch.


To formulate our main result, let us recall that the Dirichlet-to-Neumann operator is continuous, 
\begin{align*}
\Lambda_T : H_{cc}^1((0,T) \times \p M) \to L^2((0,T) \times \p M),
\end{align*}
where $H_{cc}^1((0,T) \times \p M) := \{f \in H^1((0,T) \times \p M);\ f(0, x) = 0\}$, see \cite{Lasiecka1986}.
Moreover, in Section \ref{sec_BC} 
we show that $\Lambda_{2T}$ has the additional regularity-symmetry property, 
\begin{align}
\label{additional_smoothness_Lambda}
&K(\Lambda_{2T})  
: L^2((0,T) \times \p M) \to L^2((0,T) \times \p M),
\end{align}
where $K(\Lambda_{2T}) := R \Lambda_{T} R J \Theta - J \Lambda_{2T} \Theta$,
$R$ is the time reversal on $(0, T)$, that is $Rf(t) := f(T - t)$,
$\Theta$ is the extension by zero from $(0,T)$ to $(0,2T)$ and 
\begin{align*}
&Jf(t) := \frac{1}{2} \int_t^{2T - t} f(s) ds,
\quad f \in L^2(0, 2T),\ t \in (0, T).
\end{align*}
The additional regularity can be understood
by noticing that $K(\Lambda)$ corresponds formally 
to the operator $(\Lambda^* - \Lambda)J + [\Lambda, J]$,
where $\Lambda = \Lambda_{2T}$ and $J$ are operators of orders -1 and zero, respectively.
We denote $\Upsilon := (0,T) \times \p M$ and define
\begin{align*}
\norm{\Lambda_{2T}}_* := \norm{K(\Lambda_{2T})}_{L^2(\Upsilon) \to L^2(\Upsilon)}
+ \norm{\Lambda_{T}}_{H_{cc}^1(\Upsilon) \to L^2(\Upsilon)}.
\end{align*}
We will show that $\norm{\cdot}_*$ is a norm in the appendix below.
Our main result is the following.

\begin{theorem}
\label{thm_main}
Suppose that the wave equations (\ref{eq_wave}) are stably observable
for the wave speeds in a set $U \subset C_+^\infty(M)$, from $\p M$ in time $T > 0$.
Suppose, furthermore, that there is $\epsilon_U > 0$ such that
\begin{align*}
c(x) \ge \epsilon_U, \quad \text{for all $x \in M$ and $c \in U$}.
\end{align*}
Let $R > 0$ satisfy $M \subset B(0, R)$ and let $c \in U$. 
Then there is $C > 0$ depending on
$M$, $T$, $c$, $\epsilon_U$ and $C_{obs}$
such that for all $\tilde c \in U$
\begin{align*}
|\F\ll(\tilde c^{-2} - c^{-2}\rr)(\xi)| \le C e^{2R |\xi|} \norm{\Lambda_{\tilde c, 2T} - \Lambda_{c, 2T}}_*,
\quad \xi \in \R^n,
\end{align*}
where $\F(\rho)$, $\rho \in C^\infty(M)$, 
denotes the Fourier transform of the extension by zero of $\rho$ onto $\R^n$.
\end{theorem}

\subsection{Previous literature}


We refer to the monograph \cite{Katchalov2001} and to the review article \cite{Belishev2007} for literature concerning the BC method,
and to the review article \cite{Uhlmann2009} concerning the CGO solutions.
A stability result without a modulus of continuity for the former approach 
was proved in \cite{Anderson2004} and
the first logarithmic type stability result for the latter in \cite{Alessandrini1988}.
H\"older type stability results for (IP) are proved in the above mentioned articles \cite{Stefanov1998, Stefanov2005, Bellassoued2011}.


There has been recent interest in results showing that the ill-posedness of
the inverse problem for the Helmholtz equation decreases when the frequency increases, 
see \cite{Isakov2011, Nagayasu2011} and the references therein.
Moreover, in a recent preprint \cite{Ammari2012}, Lipschitz stability for determining the low frequency component of a potential in the inverse scattering problem was established. 
This result is similar in spirit with the present one but it is based on different techniques.

As for recovering the potential in a wave equation, H\"{o}lder type stability result was first established in \cite{Sun1990}. This was improved later to an almost Lipschitz stability result with the H\"{o}lder exponent being $1-\epsilon$ in \cite{Bao2009}. 
Lipschitz type stability can be obtained if the potential is assumed to be parametrized in a finite dimensional space \cite{Rakesh1990}.
See also \cite{Alessandrini2005, Rondi2006} for a Lipschitz stability result
with finite number of parameters.

The inverse problem in the present paper is formulated with many boundary measurements. On the other hand, for a different formulation of the inverse problem with a single measurement, Lipschitz type stability results can be achieved. For example, \cite{Stefanov2011} proved Lipschitz stability for recovering the sound speed from a single measurement in the context of multi-wave imaging. 
However, such formulation typically requires non-vanishing initial data which is not in favor of practical applications if only acoustic waves are used for imaging.
The main methodology used in the inverse problems for the wave equation 
with a single measurement is based on Carleman type estimates. The technique was originated in \cite{Bukhgeuim1981} and has been developed tremendously since then. In particular, the continuous observability inequality (\ref{continuous_observability}) may be used to derive the Lipschitz type stability. For more details about the single measurement formulation, we refer to \cite{Imanuvilov2001, Isakov2006, Triggiani2011} and the references therein.

Let us also point out that the continuous observability inequality (\ref{continuous_observability}) is equivalent to the exact controllability of the wave equation, i.e. the surjectivity of the control to solution map. This well-known link to the control theory has been well studied since 1980s and we refer to the review articles \cite{Lions1988, Lasiecka2004} and the references therein for more details on the exact controllability of wave equations.   



\section{A modification of the Boundary Control method}
\label{sec_BC}

Let $\rho \in C^\infty(M)$ and let us extend $\rho$ by zero to $\R^n$.
We denote by $\F(\rho)$ the Fourier transform of the extension. 
Moreover, let us define the operator 
\def\T{\mathcal T}
\begin{align*}
&B(\Lambda_{c,T}) := R \Lambda_{c,T} R I \T_0  - I \T_1,
\end{align*}
where $\T_j$, $j=0,1$, are the first two traces on $\p M$,
that is $\T_0 \phi = \phi|_{\p M}$ and $\T_1 \phi = \p_\nu \phi|_{\p M}$,
and 
\begin{align*}
&If(t) := \int_t^{T} f(s) ds,
\quad f \in L^2(0, T),\ t \in (0, T).
\end{align*}

In this section we will prove the following reconstruction formula.
\begin{theorem}
\label{thm_rec_formula}
Let $c \in C_+^\infty(M)$, open $\Gamma \subset \p M$ and $T > 0$ satisfy the 
continuous observability inequality (\ref{continuous_observability}).
Let $\xi, \eta \in \R^n$ satisfy $|\xi| = |\eta|$ and $\xi \perp \eta$,
and define the functions
\begin{align*}
\phi_{\xi, \eta}(x) := e^{i(\xi + i \eta) \cdot x/2}, \quad
\psi_{\xi, \eta}(x) := e^{i(\xi - i \eta) \cdot x/2}.
\end{align*}
Then 
\begin{align}
\label{the_reconstruction_formula}
\F\ll(c^{-2}\rr)(\xi) 
= (K(\Lambda_{c, 2T})^\dagger B(\Lambda_{c, T}) \phi_{\xi, \eta}, 
B(\Lambda_{c, T}) \psi_{\xi, \eta})_{L^2( (0, T) \times \p M)},
\end{align}
where $K(\Lambda_{c, 2T})^\dagger$ is the pseudoinverse of $K(\Lambda_{c, 2T})$. 
\end{theorem}

We refer to \cite{Engl1996} and \cite{Ben-Israel2003}
for the definition and general theory of pseudoinverse operators.
By inspecting the proof of Theorem \ref{thm_rec_formula} we see that a formula of the type
(\ref{the_reconstruction_formula}) can be obtained given only the restriction 
\begin{align*}
\Lambda_{c,T,\Gamma} : f \mapsto \p_\nu u^f|_{(0,T) \times \Gamma},
\quad f \in C_0^{\infty}((0, T) \times \p M),
\end{align*}
of the Dirichlet-to-Neumann map. However, we will leave the partial data problems as an object of future study.
In particular, the local data problem, with the sources $f$ supported on $(0, T) \times \Gamma$,
will require further analysis. 

\subsection{Blagove{\v{s}}{\v{c}}enski{\u\i} type identities}
The BC method is based on the following identity that originates from \cite{Blagovevsvcenskiui1966},
\begin{align}
\label{eq_Blago}
(u^f(T), u^h(T))_{L^2(M; c^{-2}(x) dx)} 
= (f, K(\Lambda_{c, 2T}) h)_{L^2( (0, T) \times \p M)},
\end{align}
where $f, h \in C_0^\infty((0,T) \times \p M)$, $T > 0$
and $u^f$ denotes the solution of (\ref{eq_wave}).
The formulation of the identity (\ref{eq_Blago})
by using the operator $K(\Lambda_{c, 2T})$
is from \cite{Bingham2008}. Let us also mention that the 
iterative time-reversal control method introduced there can be adapted to 
give an efficient implementation of the reconstruction 
formula (\ref{the_reconstruction_formula}).

We define the map $W_{c,T} f := u^f(T)$ and have, see \cite{Lasiecka1986},
\begin{align*}
W_{c,T} : L^2((0,T) \times \p M) \to L^2(M).
\end{align*}
To simplify the notation, we often write $K := K(\Lambda_{c, 2T})$, $B := B(\Lambda_{c,T})$ and $W := W_{c,T}$.
Then (\ref{eq_Blago}) yields $K = W^* W$,
and $K$ extends as a continuos operator on $L^2((0,T) \times \p M)$
by the $L^2$ continuity of $W$.

\begin{lemma}
Let $f \in C_0^\infty((0,T) \times \p M)$ and let $\phi \in C^\infty(M)$ be harmonic. 
Then 
\begin{align}
\label{eq_Blago_harmonic}
(u^f(T), \phi)_{L^2(M; c^{-2}(x) dx)} 
= (f, B \phi)_{L^2((0, T) \times \p M)}.
\end{align}
In particular, $B$ is the restriction of $W^*$ on harmonic functions.
\end{lemma}
\begin{proof}
Let $t \in (0, T)$. Then integration by parts gives 
\begin{align*}
&\p_t^2 (u^f(t), \phi)_{L^2(M; c^{-2}(x) dx)} 
\\&\quad= (\Delta u^f(t), \phi)_{L^2(M; dx)} - (u^f(t), \Delta \phi)_{L^2(M; dx)}
\\&\quad= (\Lambda_T f(t), \phi)_{L^2(\p M)} - (f(t), \p_\nu \phi)_{L^2(\p M)}.
\end{align*}
By solving this differential equation with vanishing initial conditions at $t = 0$ we get 
\def\I{\mathcal I}
\begin{align*}
&(u^f(T), \phi)_{L^2(M; c^{-2}(x) dx)} 
\\&\quad=  
\int_0^T \int_0^s (\Lambda_T f(t), \phi)_{L^2(\p M)} - (f(t), \p_\nu \phi)_{L^2(\p M)}
dt ds
\\&\quad= 
(\I \Lambda_T f, \phi)_{L^2((0, T) \times \p M)} - (\I f, \p_\nu \phi)_{L^2((0, T) \times \p M)},
\end{align*}
where $\I f(s) := \int_0^s f(t) dt$.
The equation (\ref{eq_Blago_harmonic}) follows since $\Lambda_{T}^* = R \Lambda_{T} R$
and $\I^* = I$.
\end{proof}

\subsection{Computation of boundary controls}

Let $\phi \in L^2(M)$ and consider the control equation
\begin{align}
\label{eq_control}
W f = \phi, \quad \text{for $f \in L^2((0, T) \times \p M)$.}
\end{align}
Typically $W$ is not injective and 
we can hope to solve (\ref{eq_control}) 
only in the sense of the pseudoinverse $f = W^\dagger \phi$.
It is well-known that the pseudoinverse $W^\dagger$ is a bounded operator 
if and only if $W$ has closed range $R(W)$. 
If $T > 0$ is large enough then $R(W)$ is dense in $L^2(M)$
by Tataru's unique continuation \cite{Tataru1995}. 
Hence the pseudoinverse $W^\dagger$ is a bounded operator 
if and only if $W$ is surjective.
It is well-known, see e.g. \cite{Bardos1992}, that the map
\begin{align*}
f \mapsto (u^f(T), \p_t u(T)) : L^2((0,T) \times \p M) \to L^2(M) \times H^{-1}(M) 
\end{align*}
is surjective if and only if the continuous observability (\ref{continuous_observability})
holds with $\Gamma = \p M$. 
Let us now assume that (\ref{continuous_observability}) holds.
Then 
\begin{align*}
W^\dagger : L^2(M) \to L^2((0,T) \times \p M)
\end{align*}
is continuous and $W W^\dagger$ is the identity operator.

The pseudoinverse can be written as $W^\dagger = (W^* W)^\dagger W^*$.
Moreover, $R(K) = R(W^* W) = R(W^*)$ is closed since $R(W)$ is closed.
Hence $K^\dagger$ is continuous on $L^2((0, T) \times \p M)$
and $K K^\dagger$ is the orthogonal projection onto $R(K)$.
In particular, if $\phi \in C^\infty(M)$ is harmonic, then
$W^\dagger \phi = K^\dagger B \phi$ 
can be computed from the boundary data $\Lambda_{2T}$.


Let $\phi, \psi \in C^\infty(M)$ be harmonic. 
Then the identity (\ref{eq_Blago_harmonic}) yields
\begin{align*}
&(\phi, \psi)_{L^2(M; c^{-2}(x) dx)} 
= (W W^\dagger \phi, \psi)_{L^2(M; c^{-2}(x) dx)} 
\\&\quad= (W^\dagger \phi, B \psi)_{L^2( (0, T) \times \p M)}
= (K^\dagger B \phi, B \psi)_{L^2( (0, T) \times \p M)}.
\end{align*}
Notice that the functions $\phi := \phi_{\xi, \eta}$ and 
$\psi := \psi_{\xi, \eta}$
defined in Theorem \ref{thm_rec_formula}
are harmonic and $\phi(x) \psi(x) = e^{i\xi \cdot x}$. Thus
\begin{align*}
\F\ll(c^{-2}\rr)(\xi) 
= (\phi, \psi)_{L^2(M; c^{-2}(x) dx)}
= (K^\dagger B \phi, B \psi)_{L^2( (0, T) \times \p M)}.
\end{align*}
This proves Theorem \ref{thm_rec_formula}.

\section{Stability of the reconstruction}

Let us assume that $c \in C_+^\infty(M)$,
$T > 0$ and $\Gamma = \p M$ satisfy (\ref{continuous_observability}).
We denote $\Upsilon := (0,T) \times \p M$.
Notice that $W^* \phi = \p_\nu u|_{\Upsilon}$, where
$u$ is the solution of the wave equation (\ref{eq_wave_w})
with $(u, \p_t u) = (0, \phi)$ as the initial data at $t = T$.
Hence we have 
\begin{align*}
\norm{\phi}_{L^2(M)} 
\le
C_{obs} \norm{W^* \phi}_{L^2(\Upsilon)}.
\end{align*}
In particular, $W^*$ is an injection and
$(W^*)^{-1} : R(W^*) \to L^2(M)$ satisfies
$\norm{(W^*)^{-1}}_{R(W^*) \to L^2(M)} \le C_{obs}$.
Notice that 
\begin{align*}
(W^\dagger)^* = (W^*)^\dagger = (W^*)^{-1} P_{R(W^*)},
\end{align*}
where $P_{R(W^*)}$ is the orthogonal projection onto $R(W^*)$.
Hence
\begin{align*}
\norm{W^\dagger}_{L^2(M) \to L^2(\Upsilon)} 
= \norm{(W^*)^\dagger}_{L^2(\Upsilon) \to L^2(M)}
\le C_{obs}.
\end{align*}
Moreover, $(W^* W)^\dagger = W^\dagger (W^*)^\dagger$, which implies 
\begin{align*}
\norm{K^\dagger}_{L^2(\Upsilon) \to L^2(\Upsilon)}  
= \norm{(W^* W)^\dagger}_{L^2(\Upsilon) \to L^2(\Upsilon)} 
\le C_{obs}^2.
\end{align*}
We are now ready to prove Theorem \ref{thm_main} formulated in the introduction.

\begin{proof}[Proof of Theorem \ref{thm_main}]
Let $c, \tilde c \in U$ and denote
$\Lambda_{T} = \Lambda_{c, T}$,
and $\tilde \Lambda_{T} = \Lambda_{\tilde c, T}$
and let us define $W$, $\tilde W$,
$K$, $\tilde K$, $B$ and $\tilde B$ analogously.
From now on we will omit writing $L^2(\Upsilon)$ as a subscript.
We have, see e.g. \cite{Izumino1983},
\begin{align*}
\norm{\tilde K^\dagger - K^\dagger}
\le 3 \max(\norm{\tilde K^\dagger}^2, \norm{K^\dagger}^2)
\norm{\tilde K - K} \le 3 C_{obs}^4 \norm{\tilde K - K}.
\end{align*}

Notice that for $\phi \in C^\infty(M)$ the function
\begin{align*}
(R I \T_0 \phi)(t, x) = t \phi(x), \quad t \in [0, T],\ x \in \p M,
\end{align*}
is in $H_{cc}^1(\Upsilon)$.
Thus there is $C_0 > 0$ depending only on $T$ and $M$ such that
\begin{align*}
\norm{(\tilde B - B) \phi}
&= \norm{(\tilde \Lambda_T - \Lambda_T) R I \T_0 \phi}
\\&\le C_0 \norm{\tilde \Lambda_T - \Lambda_T}_{H^1_{cc}(\Upsilon) \to L^2(\Upsilon)} \norm{\phi}_{C^1(\p M)}.
\end{align*}
Moreover, if $\phi \in C^\infty(M)$ is harmonic then 
\begin{align*}
\norm{B \phi} = \norm{W^* \phi} \le C_c \norm{\phi}_{L^2(M; c^{-2} dx)},
\end{align*}
where we have denoted $C_c := \norm{W^*}_{L^2(M) \to L^2(\Upsilon)}$.

Let $\phi, \psi \in C^\infty(M)$ be harmonic. Then
\begin{align*}
&|(\phi, \psi)_{L^2(M; \tilde c^{-2} dx)} - (\phi, \psi)_{L^2(M; c^{-2} dx)}|
\le
|((\tilde K^\dagger - K^\dagger) B \phi, B \psi)|
\\&\quad\quad+
|(\tilde K^\dagger B \phi, (\tilde B - B) \psi)|
+ 
|(\tilde K^\dagger (\tilde B - B) \phi, \tilde B \psi)|.
\end{align*}
We have
\begin{align*}
&|((\tilde K^\dagger - K^\dagger) B \phi, B \psi)|
\\&\quad\le
3C_{obs}^4 C_c^2 \norm{\tilde K - K} \norm{\phi}_{L^2(M; c^{-2} dx)} \norm{\psi}_{L^2(M; c^{-2} dx)},
\\&|(\tilde K^\dagger B \phi, (\tilde B - B)\psi)|
\\&\quad\le
C_{obs}^2 C_c C_0 
\norm{\tilde \Lambda_T - \Lambda_T}_{H_{cc}^1(\Upsilon) \to L^2(\Upsilon)}
\norm{\phi}_{L^2(M; c^{-2} dx)} \norm{\psi}_{C^1(\p M)}.
\end{align*}
Notice that $(\tilde K^\dagger)^* = \tilde K^\dagger$ since $\tilde K^* = \tilde K$. 
Moreover, as $\tilde K^\dagger \tilde B \psi = \tilde W^\dagger \phi$,
 we have
\begin{align*}
&|(\tilde K^\dagger (\tilde B - B) \phi, \tilde B \psi)|
= |((\tilde B - B) \phi, \tilde W^\dagger \psi)|
\\&\quad\le
C_0 C_{obs} \norm{\tilde \Lambda_T - \Lambda_T}_{H_{cc}^1(\Upsilon) \to L^2(\Upsilon)}
\norm{\phi}_{C^1(\p M)} \norm{\psi}_{L^2(M; \tilde c^{-2} dx)}.
\end{align*}
Hence there is a constant $C > 0$ depending on $M$, $T$, $c$, $\epsilon_U$ and $C_{obs}$
such that for all $\tilde c \in U$ and harmonic $\phi, \psi \in C^1(M)$
\begin{align}
\label{stability_for_harmonic}
&|(\phi, \psi)_{L^2(M; \tilde c^{-2} dx)} - (\phi, \psi)_{L^2(M; c^{-2} dx)}|
\\\nonumber&\quad\le
C \ll( \norm{\tilde K - K} +  \norm{\tilde \Lambda_T - \Lambda_T}_{H_{cc}^1(\Upsilon) \to L^2(\Upsilon)} \rr)
\norm{\phi}_{C^1(M)} \norm{\psi}_{C^1(M)}.
\end{align}

Let $\xi, \eta \in \R^n$ satisfy
$|\xi| = |\eta|$ and define
$\phi(x) := e^{i(\xi + i \eta) \cdot x/2}$.
Then $|\phi(x)| \le e^{R |\xi|/2}$ and
\begin{align*}
|\p_{x^j} \phi(x)| \le (|\xi_j| + |\eta_j|) e^{R |\eta|/2} \le 2 |\xi| e^{R |\xi|/2}
\le C_R e^{R |\xi|},
\end{align*}
where $C_R > 0$ is a constant depending only on $R$.
Hence (\ref{stability_for_harmonic}) implies that there is $C > 0$ such that
\begin{align*}
|\F\ll(\tilde c^{-2} - c^{-2}\rr)(\xi)| \le C e^{2R |\xi|} \norm{\tilde \Lambda_{2T} - \Lambda_{2T}}_*,
\quad \xi \in \R^n,\ \tilde c \in U.
\end{align*}
\end{proof}

\section{Stable observability}
\label{sec_stable_observability}

In this section we will prove Theorem \ref{thm_stable_for_semisimples} formulated in the introduction.
As the proof is of geometric nature, we will 
consider the wave equation 
\begin{align*}
&\p_t^2 u - \Delta_{g,\mu} u = 0, &\text{in $(0, T) \times M$}, \\
 &u|_{(0, T) \times \p M} = 0, &\text{in $(0, T) \times \p M$}.
\end{align*}
on a smooth compact Riemannian manifold $(M,g)$
with boundary. Here $\Delta_{g,\mu}$ is the weighted Laplace-Beltrami operator,
\begin{align*}
\Delta_{g,\mu} u = \mu^{-1} \div_g (\mu \grad_g u),
\end{align*}
where $\mu \in C^\infty(M)$ is strictly positive
and $\div_g$ and $\grad_g$ denote the divergence and the gradient 
with respect to the metric tensor $g$.
To prove Theorem \ref{thm_stable_for_semisimples}
we will apply the results proven in this section to $g = c(x)^{-2} dx^2$
and $\mu(x) = c(x)^{n-2}$. Then $\Delta_{g,\mu} = c(x)^2 \Delta$, where $\Delta$ is the Euclidean Laplacian. 

We denote by $|\cdot|_g$, $(\cdot, \cdot)_g$, $dV_g$, $dS_g$, $\nu_g$ and $D^2_g$
the norm, the inner product,
the volume and the surface measures, the exterior unit normal vector
and the Hessian with respect to $g$,
and will omit writing the subscript $g$ when considering a fixed Riemannian metric tensor.
We recall that the Hessian satisfies
\begin{align*}
D^2 w(X, Y) = (D_X \grad w, Y), 
\quad
D^2 w(X, Y) = D^2 w(Y, X),
\end{align*}
where $w \in C^2(M)$, $X, Y \in TM$ and 
$D_X$ is the covariant derivative with respect to $g$.
We denote $\div_\mu X := \mu^{-1} \div (\mu X)$ and have the formula
\begin{align*}
\div_\mu (w X) 
= (\grad w, X) + w \div_\mu X.
\end{align*}
We will obtain a stable observability inequality from a Carleman estimate 
similar to that in \cite{Triggiani2002}.

\begin{lemma}
\label{lem_carleman_eq}
Let $\ell \in C^2(M)$, $\psi \in C^1(M)$ and let $w \in C^2(\R \times M)$.
We denote $\phi := \Delta_\mu \ell - \psi$ and $\tilde q := \phi - |\nabla \ell|^2$ and define
\begin{align*}
v := e^\ell w, \quad 
\vartheta := \psi v + 2(\grad v, \grad \ell), \quad
Y := ((\p_t v)^2 + |\grad v|^2 - \tilde q v^2) \grad \ell.
\end{align*}
Then
\begin{align}
\label{eq_carleman}
&e^{2 \ell} (\p_t^2 w - \Delta_\mu w)^2/2 
- \p_t (\vartheta\p_t v) + \div_\mu(\vartheta \grad v) + \div_\mu Y
\\\nonumber&\quad=
(\p_t^2 v - \Delta_\mu v + \tilde q v)^2/2 + \vartheta^2/2
\\\nonumber&\quad\quad
+ \phi (\p_t v)^2 - \phi |\nabla v|^2 + 2 D^2 \ell(\grad v, \grad v)
\\\nonumber&\quad\quad
+ v (\grad \psi , \grad v) + (\tilde q \psi - \div_\mu (\tilde q \grad \ell)) v^2.
\end{align}
Moreover, 
\begin{align*}
\tilde q \psi - \div_\mu (\tilde q \grad \ell)
&= \phi |\grad \ell|^2 + 2 D^2 \ell (\grad \ell, \grad \ell)
\\&\quad
+ \phi \psi - \div_\mu(\phi \grad \ell).
\end{align*}
\end{lemma}
\begin{proof}
Notice that
\begin{align*}
\Delta_\mu v &= \div_\mu( \grad (e^\ell w)) 
= \div_\mu (v \grad \ell + e^\ell \grad w)
\\&= (\grad v, \grad \ell) + v \Delta_\mu \ell + (\grad e^\ell, \grad w) + e^\ell \Delta_\mu w,
\\&= 2(\grad v, \grad \ell) +  (\Delta_\mu \ell - |\nabla \ell|^2) v + e^\ell \Delta_\mu w,
\end{align*}
since
\begin{align*}
(\grad e^\ell, \grad w) &= (e^\ell \grad \ell, \grad w) 
= (\grad \ell, \grad (e^\ell w)) - (\grad \ell, w \grad e^\ell)
\\&= (\grad \ell, \grad v) - v |\grad \ell|^2.
\end{align*}
We expand the square 
\begin{align}
\label{TY_expansion}
&e^{2\ell} (\p_t^2 w - \Delta_\mu w)^2 
= (\p_t^2 v - \Delta_\mu v  + \tilde q v + \vartheta)^2
\\\nonumber&\quad= (\p_t^2 v - \Delta_\mu v  + \tilde q v)^2 + \vartheta^2
 + 2(\p_t^2 v - \Delta_\mu v + \tilde qv)\vartheta,
\end{align}
and study the cross terms. We have
\begin{align}
\label{TY_pt}
\vartheta \p_t^2 v 
&= 
\p_t (\vartheta \p_t v) - \p_t \vartheta\, \p_t v
\\\nonumber&= 
\p_t (\vartheta \p_t v) - \psi (\p_t v)^2 - 2 \p_t v (\grad \p_t v, \grad \ell)
\\\nonumber&= 
\p_t (\vartheta \p_t v) - \div_\mu ((\p_t v)^2 \grad \ell) + (\Delta_\mu \ell - \psi)(\p_t v)^2,
\end{align}
since
\begin{align*}
2 \p_t v (\grad \p_t v, \grad \ell) = 
(\grad (\p_t v)^2, \grad \ell)
= \div_\mu ((\p_t v)^2 \grad \ell) - (\p_t v)^2 \Delta_\mu \ell.
\end{align*}
Similarly, 
\begin{align}
\label{TY_Delta}
\vartheta \Delta_\mu v 
&= \div_\mu(\vartheta \grad v) - (\grad \vartheta, \grad v)
\\\nonumber&= \div_\mu (\vartheta \grad v) - v (\grad \psi, \grad v) - \psi |\grad v|^2
- 2(\grad (\grad v, \grad \ell), \grad v),
\\\nonumber&= \div_\mu (\vartheta \grad v - |\grad v|^2 \grad \ell) - v (\grad \psi, \grad v) + (\Delta_\mu \ell - \psi) |\grad v|^2
\\\nonumber&\quad
- 2D^2 \ell(\grad v, \grad v),
\end{align}
since
\begin{align*}
2(\grad (\grad v, \grad \ell), \grad v) 
&= 2(D_{\grad v} \grad v, \grad \ell) + 2(\grad v, D_{\grad v} \grad \ell)
\\&= 2D^2 v(\grad v, \grad \ell) + 2D^2 \ell(\grad v, \grad v),
\end{align*}
and
\begin{align*}
2D^2 v(\grad v, \grad \ell) &= 2D^2 v(\grad \ell, \grad v) = 2(D_{\grad \ell} \grad v, \grad v)
= (\grad \ell, \grad |\grad v|^2)
\\&= \div_\mu(|\grad v|^2 \grad \ell) - |\grad v|^2 \Delta_\mu \ell.
\end{align*}
Finally,
\begin{align}
\label{TY_zeroth}
\vartheta \tilde q v &= \tilde q \psi v^2 + 2 \tilde q v (\grad v, \grad \ell)
\\\nonumber&= (\tilde q \psi - \div_\mu (\tilde q \grad \ell)) v^2 + \div_\mu( v^2 \tilde q \grad \ell),
\end{align}
since
\begin{align*}
&2 \tilde q v (\grad v, \grad \ell) 
= (\grad v^2, \tilde q \grad \ell) 
= \div_\mu( v^2 \tilde q \grad \ell) - v^2 \div_\mu (\tilde q \grad \ell).
\end{align*}
The first claim follows by inserting (\ref{TY_pt}), (\ref{TY_Delta}) and (\ref{TY_zeroth}) into (\ref{TY_expansion}).

For the second claim notice that
\begin{align*}
\tilde q \psi - \div_\mu (\tilde q \grad \ell)
&= (\phi - |\nabla \ell|^2)\psi - \div_\mu((\phi - |\nabla \ell|^2)\grad \ell)
\\&= (\Delta_\mu \ell - \psi) |\grad \ell|^2 + 2 D^2 \ell (\grad \ell, \grad \ell)
\\&\quad
+ \phi \psi - \div_\mu(\phi \grad \ell),
\end{align*}
since 
\begin{align*}
\div_\mu(|\grad \ell|^2 \grad \ell) = (\grad |\grad \ell|^2, \grad \ell) + |\grad \ell|^2 \Delta_\mu \ell,
\end{align*}
and
\begin{align*}
(\grad |\grad \ell|^2, \grad \ell) = 2(D_{\grad \ell} \grad \ell, \grad \ell) = 2 D^2 \ell (\grad \ell, \grad \ell).
\end{align*}
\end{proof}

\begin{corollary}[Pointwise Carleman inequality]
\label{cor_carleman_ineq}
Let $\ell \in C^3(M)$ and $\rho > 0$ satisfy
\begin{align}
\label{carleman_ineq_rho}
D^2 \ell (X, X) \ge \rho |X|^2, \quad X \in T_x M,\ x \in M.
\end{align}
Let $\tau > 0$ and $w \in C^2(\R \times M)$.
We define
\begin{align*}
&v := e^{\tau \ell} w, \quad 
\vartheta := \tau ((\Delta_\mu \ell - \rho) v + 2 (\grad v, \grad \ell)),
\\&
Y := \tau ((\p_t v)^2 + |\grad v|^2 - (\tau \rho - \tau^2 |\nabla \ell|^2) v^2) \grad \ell.
\end{align*}
Then
\begin{align*}
&e^{2 \tau\ell} (\p_t^2 w - \Delta_\mu w)^2/2 
- \p_t (\vartheta\p_t v) + \div_\mu(\vartheta \grad v) + \div_\mu Y
\\&\quad\ge
e^{2 \tau\ell}(\rho \tau - 1)((\p_t w)^2 + |\nabla w|^2)/2
+ e^{2 \tau\ell}(2 \rho |\grad \ell|^2 \tau - C_1) \tau^2 w^2,
\end{align*}
where $C_1 = \rho^2 + \max_{x \in M} |\grad (\Delta_\mu \ell(x))|^2$.
\end{corollary}
\begin{proof}
We invoke Lemma \ref{lem_carleman_eq} 
with $\ell$ replaced by $\tau \ell$ and $\psi = \tau(\Delta_\mu \ell - \rho)$.
Then $\phi= \tau \rho$.
Notice that the two first terms on the right-hand side of (\ref{eq_carleman}) are positive.
We employ (\ref{carleman_ineq_rho}) for $X = \grad v$ and for $X = \grad \ell$ to get
\begin{align*}
&e^{2 \tau \ell} (\p_t^2 w - \Delta_\mu w )^2/2 
- \p_t (\vartheta\p_t v) + \div_\mu(\vartheta \grad v) + \div_\mu Y
\\&\quad\ge
\rho \tau (\p_t v)^2 - \rho \tau |\nabla v|^2 + 2 \rho \tau |\grad v|^2
+ v (\tau \grad (\Delta_\mu \ell - \rho) , \grad v)
\\&\quad\quad
+ (\rho |\grad \ell|^2 + 2 \rho |\grad \ell|^2)\tau^3 v^2
+ (\rho(\Delta_\mu \ell - \rho) - \rho \div_\mu(\grad \ell)) \tau^2 v^2
\\&\quad\ge
\rho \tau (\p_t v)^2 + (\rho \tau - 1) |\nabla v|^2 
+ 3 \rho |\grad \ell|^2 \tau^3 v^2 
- (\rho^2 + |\grad \Delta_\mu \ell|^2) \tau^2 v^2.
\end{align*}
The claim follows by noticing that
\begin{align*}
e^{-2\tau \ell} |\grad v|^2 = |\tau w \grad \ell + \grad w|^2
\ge \frac{1}{2} |\grad w|^2 - |\grad \ell|^2 \tau^2 w^2.
\end{align*}
\end{proof}

\begin{lemma}
\label{lem_boundary_terms}
Let $T > 0$ and let $w \in C^2([0, T] \times M)$ satisfy $w(t, x) = 0$ for $(t, x) \in [0, T] \times \p M$.
Let $\tau, \rho > 0$ and let $\ell \in C^3(M)$.
We define $v$, $\vartheta$ and $Y$ as in Corollary \ref{cor_carleman_ineq}.
Moreover, we denote $dm := \mu dV$ and
\begin{align}
\label{def_Gamma}
\Gamma := \{ x \in \p M;\ (\grad \ell, \nu) > 0 \}.
\end{align}
Then 
\begin{align*}
&\int_0^T \int_M \div_\mu(\vartheta \grad v) + \div_\mu Y\, dm dt 
\le 
C_2 e^{B_\ell \tau} \tau \int_0^T \int_{\Gamma} (\p_\nu w)^2 \mu dS dt,
\end{align*}
where $B_\ell = 2\max_{x \in M} \ell(x)$ and $C_2 = 3 \max_{x \in M} |\grad \ell(x)|$. Moreover,
\begin{align*}
\int_M |\vartheta \p_t v| dm \le (C_3 \tau + C_4) e^{B_\ell \tau} \tau \int_M (\p_t w)^2 + |\grad w|^2 dm,
\end{align*}
where 
\begin{align*}
C_3 = C_F \max_{x \in M} |\grad \ell(x)|^2, \quad
C_4 = \max_{x \in M} |\grad \ell| + C_F \max_{x \in M}|\Delta_\mu \ell - \rho|/2
\end{align*}
and $C_F \ge 1$ is a constant satisfying the
Friedrichs' inequality 
\begin{align}
\label{ineq_Friedrichs}
\int_M \phi^2 dm \le C_F \int_M |\grad \phi|^2 dm, \quad \phi \in C_0^\infty(M).
\end{align}
\end{lemma}
\begin{proof}
Notice that on $[0, T] \times \p M$ we have $v = \p_t v = 0$ and 
\begin{align*}
\grad v = e^{\tau \ell} \grad w + \tau v \grad \ell = e^{\tau \ell} \grad w = e^{\tau \ell} (\grad w, \nu) \nu, 
\end{align*}
since $w$ vanishes there. Thus 
\begin{align*}
\vartheta (\grad v, \nu) + (Y, \nu)
&=  2 \tau (e^{\tau \ell}(\grad w, \nu) \nu, \grad \ell)(e^{\tau \ell}\grad w, \nu)
\\&\quad+ \tau |e^{\tau \ell} (\grad w, \nu) \nu|^2(\grad \ell, \nu)
\\&= 
3 e^{2\tau \ell} \tau (\grad w, \nu)^2 (\grad \ell, \nu).
\end{align*}
By the divergence theorem,
\begin{align*}
\int_M \div(\mu \vartheta \grad v + \mu Y) dV
= \int_{\p M} \mu (\vartheta \grad v + Y, \nu) dS,
\end{align*}
and the first claim follows.

For the second claimed inequality, notice that 
\begin{align*}
&\tau^{-1} e^{-2\tau \ell} |\vartheta \p_t v| = |((\Delta_\mu \ell - \rho) w + 2 (\grad w + \tau w \grad \ell, \grad \ell)) \p_t w|
\\&\quad\le
|(\Delta_\mu \ell - \rho) + 2 \tau |\grad \ell|^2| |w \p_t w| 
+ 2|\grad \ell| |\grad w| |\p_t w|
\\&\quad\le
|(\Delta_\mu \ell - \rho)/2 + \tau |\grad \ell|^2|(w^2 + (\p_t w)^2)
+ |\grad \ell|(|\grad w|^2 + (\p_t w)^2).
\end{align*}
Hence
\begin{align*}
&e^{-B_\ell \tau} \tau^{-1} \int_M |\vartheta \p_t v| dm 
\\&\quad\le 
(\max |\Delta_\mu \ell - \rho|/2 + \tau \max |\grad \ell|^2) \ll(\int_M w^2 dm + \int_M (\p_t w)^2 dm \rr)
\\&\quad\quad
+ \max |\grad \ell| \int_M |\grad w|^2 + (\p_t w)^2 dm,
\end{align*}
and the second claimed inequality follows from (\ref{ineq_Friedrichs}) with $\phi = w$.
\end{proof}

\begin{remark}
\label{rem_energy}
Let $u \in C^2([0, \infty) \times M)$ be a solution of 
\begin{align}
\label{eq_wave_control}
&\p_t^2 u(t,x) - \Delta_\mu u(t, x) = 0, & (t,x) \in (0,\infty) \times M,
\\\nonumber& 
u(t,x) = 0, & (t,x) \in (0,\infty) \times \p M,
\end{align}
Then the energy,
\begin{align*}
E(t) := \int_M (\p_t u(t))^2 + |\grad u(t)|^2 dm,
\end{align*}
is constant for $t \in [0, \infty)$. 
\end{remark}

We recall that the constant $C_1$ is defined in Corollary \ref{cor_carleman_ineq}
and the constants $C_2$, $C_3$, $C_4$ and $B_\ell$ are defined in Lemma \ref{lem_boundary_terms}.
Moreover, we define the constant $\beta_\ell = 2 \min_{x \in M} \ell(x)$.

\begin{theorem}[Observability inequality]
\label{thm_obs_ineq}
Suppose that there is a strictly convex function $\ell \in C^3(M)$ with no critical points.
Let $\rho, r > 0$ satisfy 
\begin{align*}
D^2 \ell (X, X) \ge \rho |X|^2, \quad |\grad \ell(x)| \ge r, 
\end{align*}
for all $X \in T_x M$ and $x \in M$,
and let $\Gamma \subset \p M$ contain the set (\ref{def_Gamma}). 
Suppose that 
\begin{align}
\label{def_T}
T > 2 (C_3 \tau + C_4) e^{(B_\ell - \beta_\ell)\tau} \tau,
\quad \text{where }
\tau = \max \ll(\frac{3}{\rho}, \frac{C_1}{2 \rho r^2} \rr).
\end{align}
Let $u \in C^2([0, T] \times M)$ be a solution of 
(\ref{eq_wave_control}).
Then
\begin{align*}
E(0) \le C(T) \int_0^T \int_{\Gamma} (\p_\nu u)^2 \mu dS dt,
\end{align*}
where 
\begin{align*}
C(T) = \frac{C_2 e^{(B_\ell - \beta_\ell)\tau} \tau}{T - 2 (C_3 \tau + C_4) e^{(B_\ell - \beta_\ell)\tau} \tau}.
\end{align*}
\end{theorem}
\begin{proof}
We will integrate the inequality of Corollary \ref{cor_carleman_ineq}.
Notice that 
\begin{align*}
\rho \tau - 1 \ge 2 \quad \text{and} \quad 2 \rho |\grad \ell|^2 \tau - C_1 \ge 0.
\end{align*}
By Lemma \ref{lem_boundary_terms} and Remark \ref{rem_energy} 
\begin{align*}
&e^{\beta_\ell \tau}
\int_0^T \int_M (\p_t u)^2 + |\nabla u|^2 dm dt
\\&\quad\le 
2 (C_3 \tau + C_4) e^{B_\ell \tau} \tau E(0)
+ C_2 e^{B_\ell \tau} \tau \int_0^T \int_{\Gamma} (\p_\nu u)^2 \mu dS dt.
\end{align*}
To conclude notice that 
\begin{align*}
\int_0^T \int_M (\p_t u)^2 + |\nabla u|^2 dm dt = T E(0).
\end{align*}
\end{proof}

\begin{corollary}[Stable observability]
\label{cor_stable_observability}
Suppose that there is a strictly convex function $\ell \in C^3(M)$ with no critical points.
Let $\rho, r > 0$ satisfy 
\begin{align*}
D^2 \ell (X, X) > \rho |X|^2, \quad |\grad \ell(x)| > r, 
\end{align*}
for all $X \in T_x M$ and $x \in M$.
Suppose that open $\Gamma \subset \p M$ satisfies
\begin{align*}
\{ x \in \p M;\ (\grad \ell, \nu) \ge 0 \} \subset \Gamma.
\end{align*}
Let $U_0$ be a bounded $C^2$ neighborhood of $(g, \mu)$.
Then there is a $C^1$ neighborhood $U$ of $g$
and constants $C, T > 0$ satisfying the following:
for all $(\tilde g, \tilde \mu) \in U_0$ such that $\tilde g \in U$,
the solutions 
\begin{align}
\label{energy_class}
\tilde u \in C([0,T]; H^1(M)) \cap C^1([0, T]; L^2(M))
\end{align}
of the wave equation,
\begin{align*}
&\p_t^2 \tilde u - \Delta_{\tilde g, \tilde \mu} \tilde u = 0,& \text{on $(0,T) \times M$},
\\\nonumber& 
\tilde u = 0, & \text{on $(0,T) \times \p M$},
\end{align*}
satisfy the observability inequality
\begin{align*}
\norm{\tilde u(0)}_{H^1_0(M)}^2 + \norm{\p_t \tilde u(0)}_{L^2(M)}^2
\le C \norm{\p_\nu \tilde u}_{L^2((0, T) \times \Gamma)}^2.
\end{align*}
\end{corollary}
\def\ss{\lambda_0}
\def\SS{\lambda_n}
\begin{proof}
Let us choose a finite number of compact coordinate neighborhoods covering $M$
and let $K$ be one of them. 
A metric $\tilde g$ is given in $K$ by a smooth matrix valued function
$\tilde g_{ij}$.
Let us denote by $\sigma(\tilde g, x)$ the smallest eigenvalue of 
the matrix $\p_{ij}^2 \ell - \tilde \Gamma_{ij}^k \p_k \ell - \rho \tilde g_{ij}$,
where $\tilde \Gamma_{ij}^k$ are  the Christoffel symbols corresponding to $\tilde g_{ij}$.
Then $\sigma : C^1 \times K \to \R$ is continuous on the compact set 
$K_0 \times K$,
where $K_0$ is the $C^1$ closure of the projection of $U_0$ on the metric tensors.
%
%
In particular, there is a $C^1$ neighborhood $U$ of $g$ such that 
$\sigma(\tilde g, x) > 0$ in $\bar U \cap K_0 \times K$.
%
That is, 
\begin{align*}
D_{\tilde g}^2 \ell (X, X) > \rho |X|_{\tilde g}^2,
\quad X \in TK,\ \tilde g \in U \cap K_0.
\end{align*}
The function $(\tilde g, x) \mapsto |\grad_{\tilde g} \ell(x)|_{\tilde g}$
is continuous on the compact set $K_0 \times K$,
whence by making $U$ smaller if necessary, we have 
\begin{align*}
|\nabla_{\tilde g} \ell(x)|_{\tilde g} > r,
\quad x \in K,\ \tilde g \in U \cap K_0.
\end{align*}
Let us suppose for a moment that $K$ intersects the set $\p M \setminus \Gamma$.
We may assume that $\p M \cap K$ is given by a defining function $F$ and that
$\tilde \nu = \grad_{\tilde g} F/ |\grad_{\tilde g} F|$.
The function 
$(\tilde g, x) \mapsto (\grad_{\tilde g} F(x), \grad_{\tilde g} \ell(x))_{\tilde g}$
is continuous on the compact set $K_0 \times K$,
whence by making $U$ smaller if necessary, we have that
\begin{align*}
(\tilde \nu(x), \grad_{\tilde g} \ell(x))_{\tilde g} < 0,
\quad x \in (K \cap \p M) \setminus \Gamma,\ \tilde g \in U \cap K_0.
\end{align*}
By taking the intersection with respect to the finite cover 
chosen in the beginning of the proof, we see that there is 
a $C^1$ neighborhood $U$ of $g$
such that all $\tilde g \in U \cap K_0$ satisfy 
the assumptions of Theorem \ref{thm_obs_ineq}
with the fixed $\ell$, $\rho$, $r$ and $\Gamma$.

Let us show next that the constants $C_j(\tilde g, \tilde \mu)$, $j = 1, 2, 3, 4$,
stay bounded in $U_0$.
We may first work locally in a compact coordinate neighborhood $K$ as above.
Let us denote by $\ss(\tilde g, x)$ and $\SS(\tilde g, x)$ 
the smallest and the largest eigenvalue of $\tilde g_{ij}(x)$.
The functions $\ss$ and $\SS$ are continuous 
$C^0 \times K \to \R$,
$\tilde g_{ij}(x)$ is positive definite and $\bar U_0 \times K$
is compact in $C^0 \times K$, where the closure is in $C^0$.
In particular, we may choose $C > 0$ so that on $\bar U_0$
\begin{align*}
C^{-1}|X|_{g} \le |X|_{\tilde g} \le C |X|_{g},
\quad X \in TK.
\end{align*}

Moreover, we may choose $C > 0$ so that also
the functions $|\tilde g(x)|$
and $\tilde \mu(x)$
are bounded below by $C^{-1}$ and above by $C$
on $\bar U_0 \times K$.
Hence there is $C_F > 0$ such that 
the Friedrichs' inequality 
\begin{align*}
\int_M \phi^2 d\tilde m \le C_F \int_M |\grad \phi|_{\tilde g}^2 d\tilde m, \quad \phi \in C_0^\infty(M),
\end{align*}
holds for all $(\tilde g, \tilde \mu) \in \bar U_0$.
Now it straightforward to see that 
$C_j(\tilde g, \tilde \mu)$, $j = 1, 2, 3, 4$, are bounded on $U_0$ as they can be expressed
in coordinates using the derivatives of $\tilde g_{ij}$ and $\tilde \mu$ up to the second order. 

The map $(\tilde u(0), \p_t \tilde u(0)) \mapsto \p_\nu \tilde u|_{(0,T) \times \Gamma}$
is continuous from $H_0^1(M) \times L^2(M)$ to $L^2((0,T) \times \Gamma)$ by \cite{Lasiecka1986}. Thus we may approximate the initial data by smooth compactly supported functions and get the observability also for solutions in the energy class (\ref{energy_class}).
\end{proof}

Theorem \ref{thm_stable_for_semisimples}
follows from Corollary \ref{cor_stable_observability}
by choosing $g = c(x)^{-2} dx^2$
and $\mu(x) = c(x)^{n-2}$. 

%
%

\section*{Appendix: A linear space for Dirichlet-to-Neumann operators}
\def\K{\mathcal K}

Let us consider the operator
\begin{align*}
A : H^1_{cc}((0, T) \times \p M) \to L^2((0, 2T) \times \p M),
\end{align*}
and define the map $\K(A) := R A_T R J \Theta - J A$
and the restrictions
\begin{align*}
A_T f := (Af)|_{(0, T) \times \p M}, 
\quad 
A_R f := (Af)|_{(T, 2T) \times \p M}.
\end{align*}
Notice that $2 J \Theta f(t) = \int_t^T f(s) ds$,
whence $R J \Theta f(0) = 0$ and 
$\K(A)$ is well-defined on $H^1_{cc}((0, T) \times \p M)$.
We define  
\begin{align*}
\norm{A}_* := \norm{\K(A)}_{L^2(\Upsilon) \to L^2(\Upsilon)}
+ \norm{A_T}_{H_{cc}^1(\Upsilon) \to L^2(\Upsilon)},
\end{align*}
where we have denoted $\Upsilon := (0, T) \times \p M$.
Let us next show that $\norm{\cdot}_*$ is a norm on 
\begin{align*}
H^1_{DN} := \{ A : H^1_{cc}(\Upsilon) \to L^2((0, 2T) \times \p M);\ \norm{A}_* < \infty \}.
\end{align*}
Clearly, $\norm{\cdot}_*$ is homogeneous and subadditive.
Suppose that $\norm{A}_* = 0$
and let $f \in H^1_{cc}(\Upsilon)$.
Then $A_T = 0$ and 
\begin{align*}
0 = -2\K(A)f(t) = 2JAf(t) = \int_T^{2T - t} A_R f(s) ds, \quad t \in (0, T).
\end{align*}
By differentiating, we see that $A_R f(2T - t) = 0$
for almost all $t \in (0, T)$,
whence $A_R = 0$ and we have shown that $\norm{\cdot}_*$ is a norm.

Let us now consider the Dirichlet-to-Neumann operator
\begin{align*}
\Lambda_{2T} : H^1_{cc}((0, 2T) \times \p M) \to L^2((0, 2T) \times \p M).
\end{align*}
Let 
$E : H^1_{cc}(0, T) \to H^1_{cc}(0, 2T)$
be an extension operator, that is, $Ef|_{[0, T]} = f$.
Then
\begin{align*}
\Lambda_{2T} \circ E : H^1_{cc}((0, T)\times \p M) \to L^2((0, 2T) \times \p M).
\end{align*}
Moreover, causality of the wave equation (\ref{eq_wave}) yields
$(\Lambda_{2T} \circ E)_T = \Lambda_T$ and
\begin{align*}
(h, \K(\Lambda_{2T} \circ E) f) = (u^h(T), u^f(T))_{L^2(M; c^{-2} dx)}
= (h, K(\Lambda_{2T}) f),
\end{align*}
for all $f, h \in C_0^\infty((0, T) \times \p M)$.
Hence the embedding $\Lambda_{2T} \mapsto \Lambda_{2T} \circ E$
of the Dirichlet-to-Neumann operators to $H^1_{DN}$
does not depend on the choice of the extension operator $E$.

\bigskip
{\em Acknowledgements.}
The research was supported by Finnish Centre of Excellence in Inverse
Problems Research, Academy of Finland project COE 250215,
Academy of Finland project 141075 and European Research Council advanced grant 400803.

\bibliographystyle{abbrv} 
\bibliography{main-shitao}
\end{document}